\documentclass{article}

\usepackage{authblk}

\usepackage{amssymb}
\usepackage{amsmath}
\usepackage{amsthm}
\usepackage[mathscr]{eucal}

\usepackage[active]{srcltx}

\newtheorem{thrm}{Theorem}[section]
\newtheorem{lemma}[thrm]{Lemma}

\newtheorem{cor}[thrm]{Corollary}

\newtheorem*{main*}{Main Theorem}

\theoremstyle{definition}

\theoremstyle{remark}

\numberwithin{equation}{section}

\usepackage[utf8]{inputenc}

\newcommand{\dbar}{$\bar{\partial}$}
\newcommand{\mdbar}{\bar{\partial}}

\newcommand{\D}{\Omega}

\newcommand{\zb}{\bar{z}}

\newcommand{\Lb}{\overline{L}}

\makeatletter
\newcommand{\opnorm}{\@ifstar\@opnorms\@opnorm}
\newcommand{\@opnorms}[1]{%
  \left|\mkern-1.5mu\left|\mkern-1.5mu\left|
   #1
  \right|\mkern-1.5mu\right|\mkern-1.5mu\right|
}
\newcommand{\@opnorm}[2][]{%
  \mathopen{#1|\mkern-1.5mu#1|\mkern-1.5mu#1|}
  #2
  \mathclose{#1|\mkern-1.5mu#1|\mkern-1.5mu#1|}
} \makeatother

\begin{document}

\title
 {Subelliptic estimates
 	for the \dbar-problem on complex algebraic surfaces with isolated singularities}

\author{Dariush Ehsani}

\affil{ FIZ Karlsruhe - Leibniz Institute for Information Infrastructure\\
	Department of Mathematics\\D-10587 Berlin\\
\textit{E-mail address}: \texttt{dehsani.math@gmail.com}
\vskip.1cm\
}

\date{}

\maketitle

\bibliographystyle{plain}

\begin{abstract}
	We obtain subelliptic estimates
	for the \dbar-problem on complex algebraic surfaces embedded in $\mathbb{C}^n$ with isolated singularities. 
	$W^{\epsilon}$ Sobolev norms of a form,
	$f$, for $0< \epsilon < 1$
	are estimated in terms of weighted $L^2$ norms of 
	$\mdbar f$ and $\mdbar^{\ast}f$, with weights
	 which vanish at the singularities,
	  as well as weighted 
	$L^2$ norms of $f$, with weights which blow up at the
	 singularities.
\end{abstract}

\section{Introduction}

We let $X$ be an algebraic surface over $\mathbb{C}$
 embedded in $\mathbb{C}^n$ with isolated singularities.  
The main goal of this article is to obtain subelliptic estimates 
for the \dbar-problem on $X$.  Subelliptic estimates 
 are an important topic in the theory of the 
\dbar-Neumann problem, and in particular 
 provide regularity of the solution to the 
 \dbar-problem.  Whereas on smooth domains
 where such estimates are related to the geometry of 
the boundary, on complex spaces with isolated 
 singularities little is known about the regularity of the
\dbar-problem.  There has however been significant progress
 made in the study of the $L^2$-cohomology for the \dbar-operator,
\cite{OV13} (see also \cite{Rup19}).

A study of 
 subellitptic estimates on complex spaces with isolated 
  singularities was initiated in \cite{ER} 
 with the example $z^2=xy$ in $\mathbb{C}^3$.
The current article builds off the idea of the work on 
 that example but achieves considerably more
  generality.  Namely, we obtain results which apply to all complex algebraic surfaces with isolated 
 singularities.  

We work with some of the following simplifications, without any loss 
 of generality.  As the theory of the \dbar-Neumann problem is 
well established on smooth manifolds, we will work in a 
 neighborhood of an isolated singularity.  We assume the
singularity lies at the origin and $D^n(1) \cap (X-0)$,
 where $D^n(1)$ is the ball of radius 1,
  contains no other singularities.

For our main result, we let $f$ be a $(p,q)$-form
($0\le p,q \le 2$) with support in 
$D^n(1) \cap X$ and with
 $ f\in  \mbox{dom}(\mdbar)
 \cap \mbox{dom}(\mdbar^{\ast})$,
that is 
 $\mdbar f  \in L_{(p,q+1)}^2(X)$ and 
 $\mdbar^{\ast} f \in L_{(p,q-1)}^2(X)$.  
For some $0<\epsilon \le 1$,
  we also suppose
$(r^{\epsilon}\log(r))^{-1} f \in L_{(p,q)}^2(X)$,
where $r = |z|$.   We will often drop the
 designation of the form type in the notation 
 of the Sobolev spaces; thus, $L^2(X)$ will also stand,
for instance, for $L^2_{(0,1)}(X)$ where appropriate. 
 We 
establish the
\begin{main*}
	 For $f$ as above we have 
$f\in W^{\epsilon}(X)$ with the estimates
\begin{equation*}
\|f\|_{W^{\epsilon}(X)}
\lesssim
\left\|  r^{1-\epsilon} \mdbar f\right\|_{L^2(X)} +
\left\| r^{1-\epsilon}\mdbar^{\ast} f\right\|_{L^2(X)}
+\left\|  \frac{1}{r^{\epsilon} \log(r)} f \right\|_{L^2(X)}.
\end{equation*}
\end{main*}

The intermediate Sobolev norms are defined
 by interpolation
(see \cite{BL}), and the first step in the proof
 is to establish $W^1(X)$ estimates (so that interpolation 
  can follow).  The works of
\cite{HsPa} and the related \cite{Na89}
 are essential in our use of
 coordinates which are particularly
helpful in establishing estimates for the
 intermediate norms.  

The author wishes to acknowledge and express 
 sincere gratitude for   
many fruitful mathematical exchanges with
 Jean Ruppenthal.  These discussions began 
during the author's employment 
with the Complex Analysis Group
 at the University of Wuppertal
and contributed in essential ways towards the completion 
 of this work.

\section{Sobolev 1 estimate}
\label{sectionSobolev1}

If we first assume that a 
smooth form, $f$, is supported in
a neighborhood of the origin, but away from the singularity at 0,
we can follow the establishment of the 
 Morrey-Kohn-H\"{o}mander identity
(see for instance \cite{CS} 
Proposition 4.3.1 and Proposition 5.1.1), using 
integration by parts to show
\begin{equation}
\label{estWithConstants}
\|  f\|_{W^1(X)}^2 \le C_1 \left(\| \mdbar f\|^2_{L^2(X)} +
\| \mdbar^{\ast} f\|^2_{L^2(X)}\right)
+ C_2 \|f\|_{L^2(X)}.
\end{equation}
   The constants,
$C_1$ and  $C_2$, in the
above inequality, however, may depend on $f$;
for instance, it is not known beforehand that
 a finite number of charts 
 suffices to cover a neighborhood 
  of the singularity, the constants of the
 above inequality depending on
  derivatives of local cutoffs subordinate to
 the charts.

Our strategy is to cover 
$D^n(1) \cap (X-0)$,
with neighborhoods such that
each neighborhood is contained
 in a chart (to which we will refer here as a {\it resolution chart}) 
  obtained by a resolution of the singularity
at the origin. 
A resolution 
leads to a finite number of charts which cover
$U = D^n(1)\cap(X-0)$. 
The charts are of the form
\begin{equation}
\label{charts}
\begin{aligned}
&z_1 = u^{n_1}v^{m_1},\\
&z_2 = f_2(z_1) + u^{n_2}v^{m_2},\\
&z_i = f_i(z_1) + u^{n_i}v^{m_i}g_i(u,v),
\end{aligned}
\end{equation}
where $n_i\ge n_j$, $m_i\ge m_j$ for
$i>j$, with $n_1 m_i - n_i m_1 \neq 0$ for
$i\neq 1$,
$f_i$ and $f_i'$ are holomorphic functions
of $u$ and $v$, and $f_i = o(z_1)$, and
the $g_i$ are local units (are holomorphic with
non-zero constant terms in their series expansions),
see \cite{HsPa}.
We refer the reader to \cite{Lau} for
background information on the resolution 
of singularities via quadratic transformations.

Set $N = \partial U = \partial D^n(1)\cap (X-0)$.
From \cite{HsPa}, there is a
piecewise smooth diffeomorphism,
\begin{equation}
\label{diffeoMap}
h: N \times (0,1] \rightarrow D^n(1)\cap(X-0).
\end{equation}
The diffeomorphism is obtained by 
covering $\partial U$ with neighborhoods, and in each neighborhood
following flow lines from points on $\partial U$ to the origin
(the flow lines are piecewise smooth). 
 The covering of 
$\partial U$ is chosen so that 
 each neighborhood and its trace along flow lines to the origin 
 is contained in a coordinate chart obtained by resolving
 the singularity at 0 with
 repeated quadratic transformations, and so the traces along
flow lines of the covering of $\partial U$ constitutes a covering
 of $U$, written as $U= \cup_i U_i$.  From above, 
 $U_i$ is covered by a resolution chart, and
we can write $U_i = \cup_{j=1}^{j_i} U_{ij}$, where the
 $U_{ij}$ are the regions on which 
  $h^{-1}$ is smooth,
with the property that
  $U_{i j_k} \cap U_{ij_l} 
 = \partial U_{i j_k}
  \cap \partial U_{ij_l}$
 for $k\neq l$.   Since the $U_i$ consititute a finite
  covering,
we have
\begin{equation*}
 \| f \|_{W^1(U)}
 \simeq \sum_i \| f \|_{W^1(U_i)},
\end{equation*}
and
if we write
$L_1^i$ and $L_2^i$ for the holomorphic vector fields 
 on $U_i$,
we can then write for a function, $f$,
\begin{equation*}
\| f \|_{W^1(U_i)}
\simeq \sum_{j=1}^{j_i}\sum_{k=1}^2 \| L_k^{i}
f
\|_{L^2({U}_{ij})} 
+  \| \Lb_k^{i} f \|_{L^2({U}_{ij})} + \| f\|_{L^2(U_i)}.
\end{equation*}
We thus have
\begin{equation}
\label{sob1L1L2}
\| f \|_{W^1(U)}
\simeq \sum_{i}\sum_{j=1}^{j_i}\sum_{k=1}^2 \| L_k^{i}
f
\|_{L^2({U}_{ij})} 
+  \| \Lb_k^{i} f \|_{L^2({U}_{ij})} + \| f\|_{L^2(U)}.
\end{equation}


Using a family of smooth cutoff functions,
$\{\varphi_i\}_i$ subordinate 
 to the above covering of $\partial U$, we then extend
the cutoff functions along the flow lines in each
 $U_i$.  With this construction,
 each $\varphi_i$ will depend only on variables
$\theta^i,x^i,y^i$ on 
$N = \partial U$ (see also the next section),
 and derivatives of $\varphi_i$ are bounded.
  We can thus write
\begin{equation*}
	\| f \|_{W^1(U)}
	\simeq \sum_{i,j} \| \varphi_i f \|_{W^1(U_{ij})}.
\end{equation*}
In what follows, we shall write
 $f^i$ for $\varphi_i f$. 
 The $\varphi_i$ functions allow us to assume without
loss of generality that $\mbox{supp}(f^i) \cap \partial U_i
 \subset \{0\}$, which we shall do in the next paragraph.

We now return to \eqref{estWithConstants}, which was
 obtained with the assumption of support away from the singularity.
In order to obtain estimates for 
all $f$, without restricting the support away from 0, we use 
cutoffs $\mu_k$ as in \cite{PS91}.  $\mu_k = \mu_k(|z|)$ is a smooth function
  with the property $\mu_k = 0$ for 
$|z| <   e^{-e^{k+1}}$ and 
$\mu_k = 1$ for $|z| > e^{-e^{k}}$.
  Furthermore,
\begin{equation*}
| d \mu_k(z) | \lesssim \frac{ \chi_k (|z|^2)}
{|z| \log |z| },
\end{equation*}
where $\chi_k$ is the characteristic function of
$[e^{-e^{k+1}},e^{-e^{k}}]$.  From 
\eqref{sob1L1L2}, we have
\begin{equation*}
\|  \mu_l f \|_{W^1(U)}^2 \simeq 
\sum_{i}\sum_{j=1}^{j_i}\sum_{k=1}^2 
\| L_k^{i} \mu_l f^i \|_{L^2({U}_{ij})} 
+  \| \Lb_k^{i} \mu_l f^i \|_{L^2({U}_{ij})} + \| \mu_l f\|_{L^2(U)}.
\end{equation*}
The above inequalities for Sobolev norms on functions 
 can be applied to the case of forms component-wise. 
Then following the proof of 
the Morrey-Kohn-H\"{o}mander identity
(see \cite{CS} Proposition 4.3.1 and Proposition 5.1.1), 
we can integrate by parts
 in an expression for
$\| \mdbar (\mu_k f^i)\|^2_{L^2(X)} 
+ \|\mdbar^{\ast}(\mu_k f^i)\|^2_{L^2(X)}$,
 taking into account
 the fact that the (non-zero) boundary integrals which arise along
$\partial U_{ij}$ all cancel, then sum over the $U_i$ neighborhoods,
  and write
\begin{equation}
\label{MKHwithMu_k}
\begin{aligned}
\|  \mu_k f \|_{W^1(X)}^2 \lesssim& \|  \mdbar (\mu_k f)\|^2_{L^2(X)} +
\|  \mdbar^{\ast}  (\mu_k f)\|^2_{L^2(X)}
+  \left\|  \mu_k f \right\|^2_{L^2(X)}\\
\lesssim & 
\| \mu_k  \mdbar f\|^2_{L^2(X)} +
\|  \mu_k  \mdbar^{\ast} f\|^2_{L^2(X)}
+  \left\| \frac{\chi_k(r^2)}{r\log(r)} f \right\|^2_{L^2(X)},
\end{aligned}
\end{equation}
which also takes into account estimates away from the singularity
 (obtained by classical estimates on smooth manifolds).

Define $X_k := X \cap \mbox{supp}(\mu_k)$.
Note that
\begin{equation*}
\|  f \|_{W^1(X_k)}^2 \lesssim 
\| \mu_{k+1} f \|_{W^1(X)}^2  
\end{equation*}
and
 for $f\in W^1(X)$, 
\begin{equation*}
\|  f \|_{W^1(X)}^2  =  \lim_{k\rightarrow \infty} \|  f \|_{W^1(X_k)}^2  
\end{equation*}
so that
letting $k\rightarrow\infty$ in \eqref{MKHwithMu_k} leads to
\begin{equation}
\label{basicEstLog}
\|  f \|_{W^1(X)}^2 \lesssim  \left(\| \mdbar f\|_{L^2(X)}^2 +
\|  \mdbar^{\ast} f\|_{L^2(X)}^2\right)
+  \left\| \frac{1}{r\log(r)} f \right\|_{L^2(X)}.
\end{equation}

\section{Useful coordinates}
\label{sectionCoorSys}

In \eqref{sob1L1L2}, we estimate 
 the $W^1(X)$ norm by finding holomorphic 
vector fields and estimating $L^2$ norms of the vector 
 fields (and their conjugates) acting on a function (or form).
In this section we use quasi-isometries proved in
\cite{HsPa} and updated in 
 \cite{Na88} and \cite{Na89}
to define an equivlent norm
 to 
that in \eqref{sob1L1L2} 
which will aid in our comparison
  of intermediate (between 0 and 1) Sobolev norms with that
 of $W^1(X)$.

With
\begin{equation*}
r = \left( \sum_{i=1}^n |z_i|^2 \right)^{1/2},
\end{equation*}
a coordinate system over the region $U_i$ may
  be chosen in which $r$ is one coordinate and
local coordinates  $\theta^i,x^i,y^i$ on 
$N = \partial U$ form the other. 
 This is the coordinate system used to describe the
diffeomorphism, $h$, in \eqref{diffeoMap}.  The pullback of
\begin{equation}
\label{eucMetric}
dz_1 d\zb_1+ dz_2d\zb_2+ \cdots dz_n d\zb_n,
\end{equation}
under $h$ is quasi-isometric to either a metric of 
Cheeger type \cite{HsPa},
\begin{equation}
\label{Cheeger}
 dr^2 + r^2(d\theta^i)^2 + r^{2\alpha_i}\left( (dx^i)^2 + (dy^i)^2 \right),
\end{equation}
in a region $U^i$, where $\alpha_i \ge 1$, or a metric of the form
\begin{equation}
\label{NagaseMetric}
dr^2 + r^2(d\theta^i)^2 + r^{2\alpha_i}
 \left( (ds^i)^2 + (r^{b_i} +(s^i))^2 (d\varphi^i)^2\right)
\end{equation}
\cite{Na89}
 (with $x^i = s^i\cos(\varphi^i)$, $ y^i = s^i\sin(\varphi^i) $)
 for $0<b_i<1$.

The respective volume forms over a 
 region $U_i$ have the property
\begin{equation}
\label{volumeHP}
dV^i
\simeq  r^{2\alpha_i + 1} dr d\theta^i dx^i dy^i,
\end{equation}
using \eqref{Cheeger},
or
\begin{equation}
\label{volumeNa}
dV^i
\simeq  r^{2\alpha_i + 1} (r^{b_i}+s^i) dr d\theta^i ds^i d\varphi^i,
\end{equation}
using \eqref{NagaseMetric}.

We can thus estimate $\|f\|_{W^1(X)}$ as defined by
 \eqref{sob1L1L2} with the use of vector fields
\begin{equation}
\label{vfEasy}
  \partial_r, \frac{1}{r}\partial_{\theta^i}, 
   \frac{1}{r^{\alpha_i}}\partial_{x^i},  \frac{1}{r^{\alpha_i}}\partial_{y^i}
\end{equation}
and the volume form
 in \eqref{volumeHP}, respectively
\begin{equation}
\label{vfNa}
\partial_r, \frac{1}{r}\partial_{\theta^i}, 
\frac{1}{r^{\alpha_i}}\partial_{s^i},  \frac{1}{r^{\alpha_i}}
 \frac{1}{r^{b_i}+s^i} \partial_{\varphi^i}
\end{equation}
and the volume form
in \eqref{volumeNa},
  over a region,
 $U_i$.

\section{Intermediate Sobolev spaces}

In this section we look at intermediate Sobolev spaces,
 between $L^2(X)$ and $W^1(X)$, where $W^1(X)$ 
 is defined in \eqref{sob1L1L2}.
We obtain estimates for the intermediate spaces in terms of
$L^2(X)$ and $W^1(X)$ norms.  

Recall that $U = D^n(1)\cap (X-0)$, which is diffeomorphic,
via $h$ in \eqref{diffeoMap}, to $\partial U\times(0,1]$.
We note that
\begin{lemma}
	\label{lemmaDensity}
	$W^1(U)$ is dense in $L^2(U)$.
\end{lemma}
\begin{proof}	
	
	For any $\varepsilon >0$,
	denote $U_{\varepsilon} = h(N\times(\varepsilon,1))$.
The lemma follows from the density of $C^{\infty}_0(U_{\varepsilon})$
 in $L^2(U_{\varepsilon})$ ($U_{\varepsilon}$ is a complex 
manifold, without singularities) and the density of 
$L^2(U_{\varepsilon})$  in $L^2(U)$, using extensions by zero 
 to relate a function in $L^2(U_{\varepsilon})$ to one in $L^2(U)$.
\end{proof}

Using the density of $W^1(U)$ in $L^2(U)$, there exists
 a positive, self-adjoint operator, $\Lambda$,
such that 
\begin{equation*}
 (u,v)_{W^1(U)} = (\Lambda u, \Lambda v)_{L^2(U)} \qquad \forall u,v \in W^1(U).
\end{equation*} 
Intermediate Sobolev norms can be defined 
 in terms of powers of the operator $\Lambda$ as in
  \cite{LiMa} (Section 2.1).  We denote these spaces by $W^{\epsilon}(U)$:
$W^{\epsilon}(U) = \mbox{dom}(\Lambda^{\epsilon})$ for
 $0\le \epsilon \le 1$, with norm
\begin{equation*}
 \| f\|_{W^{\epsilon}(U)} = \|f\|_{L^2(U)} +  \|\Lambda^{\epsilon} f\|_{L^2(U)} .
\end{equation*}

There are some techniques which allow one to interpolate between the spaces.
 Two methods, the "$K$-method" and the "$J$-method," can be 
  used to interpolate between $L^2(U)$ and $W^1(U)$, each 
with equivalent norms.  We refer the reader to \cite{BL} for 
a description of these methods.
  From Theorem 15.1 in \cite{LiMa} 
 the spaces $W^{\epsilon}(U)$ as defined above are equivalent to
  those produced by using
 interpolation via the $K$-method, and give equivalent norms.  Furthermore, 
  we can use a norm described as in \cite{BL} to calculate the $K$-method
   norms.  
To calculate the $W^{\epsilon}$-norm ($0<\epsilon<1$) of a function,
 $f$, known to be in $W^1(U)$, we will use that
  $\|f\|_{W^{\epsilon}(U)}$ is
 comparable to the infimum of the norm
\begin{multline}
I(u)=
\\
\left( \int_0^{\infty}
\left\| t^{1-\epsilon} u(t)
\right\|^2_{W^1(X)} t^{-1}
dt\right)^{1/2} + \left(\int_0^{\infty}
\left\| t^{1-\epsilon} u'(t)
\right\|^2_{L^2(X)} t^{-1}
dt\right)^{1/2}
\label{defnIu}
\end{multline}
taken over all
\begin{equation*}
u:[0,\infty) \rightarrow L^2(X)
+ W^1(X)
\end{equation*}
with $u(0)=f$ such that $u(t)$ is locally integrable 
in $W^1(X)$, and $u'(t)$, defined in the distributional
sense, is locally integrable in $L^2(X)$, and such that
the two terms on the right-hand side of 
\eqref{defnIu} are finite (see Section 3.12, in particular
Theorem 3.12.2, in \cite{BL}).

We use this to show
the
\begin{thrm}
	\label{ThrmSubellIneq}
	Let $X$ be a complex surface with singularities.  Let
	$W^1(X)$ be defined by \eqref{sob1L1L2}, and the
intermediate Sobolev spaces $W^{\epsilon}(X)$ as above.
Then for $f$ a smooth function in 
$U=D^n(1)\cap(X-0)$, we have the estimates
	\begin{equation*}
	\|f\|_{W^{\epsilon}(X)}
	\lesssim
	\| r^{1-\epsilon} f\|_{W^1(X)} +
	\| r^{-\epsilon} f\|_{L^2(X)},
	\end{equation*}
for $0<\epsilon<1$.
\end{thrm}

The estimate is obviously true for functions supported away from
a singularity and so we prove Theorem \ref{ThrmSubellIneq} for functions
supported in a neighborhood of the singularity.

\begin{proof}

As in Section \ref{sectionSobolev1}, we assume
 $x=0 \in X$ is a singular point, and 
there are no other singularities in $U$.
We write 
 $U = \cup_i U_i$, where $U_i$ is covered by a resolution chart.  Further recall that we write
$U_i = \cup_{j=1}^{j_i} U_{ij}$, where the
$U_{ij}$ are the regions on which 
$h^{-1}$ is smooth, where $h$ is the 
diffeomorphism from 
\eqref{diffeoMap}
 
 We assume 
$f$ is supported in $D^n(s)\cap X$
 for some $s<1$.
 Over $U_i$ we write
$f^i:=f\big|_{U_i}$ in terms of coordinates
 $r,{\theta}^i,
 {x}^i,{y}^i$:  
$f^i=f^i(r,{\theta}^i,
{x}^i,{y}^i)$. 
Define
\begin{align*}
u^i(t) &= f^i(r+t,{\theta}^i,
{x}^i,{y}^i)\eta(t)\\
&:= f^i_t\eta(t),
\end{align*}
where $\eta(t)\in C^{\infty}_0(\overline{\mathbb{R}}^+)$ with the
properties $\eta(t) \equiv 1$ near $t=0$ and $\eta(t)\equiv 0$
for $t>\delta$, where $\delta$ is chosen small enough so
that the support of $f$ is contained in 
$h(N\times (0,1-\delta])$.  Note
that $u^i(0)=f^i$ holds.  
We also use the properties of the
 volume form 
$dV^i$ given in \eqref{volumeHP}
 and \eqref{volumeNa}.

For the first term on the right-hand side of
\eqref{defnIu}, we calculate
\begin{equation}
\begin{aligned}
\label{t1-eNorm}
\int_0^{\infty}
\left\| t^{1-\epsilon} u^i(t)
\right\|^2_{W^1(X)} t^{-1}
dt&\\
\lesssim&
\int_0^{\delta}
t^{1-2\epsilon}
\sum_{i,j,k}  \int_{{U}_{ij}} 
\left(
|L_k^i  u^i |^2 +|\Lb_k^i  u^i |^2
\right) dV^i dt\\
& + \int_0^{\delta}
t^{1-2\epsilon}
\left\| u^i
\right\|^2_{L^2(X)}  dt ,
\end{aligned}
\end{equation}
where $L_1^i$ and $L_2^i$ are as in 
 Section \ref{sectionSobolev1}, \eqref{sob1L1L2}.

Let $U_i^{\sigma}$ denote the slice
 $U_i \cap h(N\times\{\sigma\})$ and
  $U_{ij}^{\sigma} = U_{ij}\cap U_i^{\sigma}$.
Thus, for instance, $U_i^1$ is $U_i \cap \partial U$.
We first handle the case of 
\eqref{volumeHP}.  We separate the integrals over the tangential components
using the properties of \eqref{volumeHP} and \eqref{volumeNa} 
for the volume element:
\begin{multline}
\label{normalTangDecomp}
 \int_{{U}_{ij}} 
\bigg(
|L_k^i  u^i |^2 +|\Lb_k^i  u^i |^2
\bigg) dV^i \\
\simeq 
 \int_0^1
 \|\partial_{r} u^i \|_{L^2(U_{ij}^1)}^2 
 r^{2\alpha_i+1} dr 
+ \int_0^1
 \| u^i \|_{W^1(U_{ij}^1)}^2 
 r^{2\alpha_i+1} dr.
\end{multline}

We make a change of variables
$r'=r+t$, and with this 
change we estimate
\begin{align*}
\int_{0}^{\delta}\int_0^1
 \|\partial_{r} u^i\|_{L^2(U_{ij}^1)}^2 
 &r^{2\alpha_i+1} t^{1-2\epsilon} dr dt\\
\lesssim&
\int_{0}^{\delta} \eta^2(t) \int_t^1
\|\partial_{r'} f^i \|_{L^2(U_{ij}^1)}^2 
(r'-t)^{2\alpha_i+1} t^{1-2\epsilon} dr' dt.
\end{align*}
We now change the order of the 
$r'$ and $t$ integrations, and estimate
\begin{multline}
\label{dr'}
\int_0^{\delta}
\int_0^{r'}     
  \|\partial_{r'} f^i \|_{L^2(U_{ij}^1)}^2 
(r'-t)^{2\alpha_i+1}
\eta^2(t) t^{1-2\epsilon} dt d{r}'\\
+ 
\int_{\delta}^1
\int_0^{\delta}  
\|\partial_{r'} f^i \|_{L^2(U_{ij}^1)}^2 
 (r'-t)^{2\alpha_i+1}
\eta^2(t) t^{1-2\epsilon} dt d{r}'.
\end{multline}
We use
\begin{equation*}
\int_0^{r'}
(r'-t)^{2\alpha_i+1}
\eta^2(t) t^{1-2\epsilon} 
dt \lesssim (r')^{2\alpha_i + 3-2\epsilon}
\end{equation*}
and
\begin{align*}
\int_0^{\delta}
(r'-t)^{2\alpha_i+1}
\eta^2(t) t^{1-2\epsilon} 
dt \lesssim& (r')^{2\alpha_i + 3-2\epsilon}
\end{align*}
since $\delta<r'$ in the second integral in \eqref{dr'}.

The integrals in \eqref{dr'} can now be bounded by
\begin{equation*}
\int_0^{1} 
\|\partial_{r'} f^i \|_{L^2(U_{ij}^1)}^2 
(r')^{2\alpha_i + 3-2\epsilon}
 d{r}'
  =\int_0^{1} 
  \|\partial_{r} f^i \|_{L^2(U_{ij}^1)}^2 
  r^{2\alpha_i + 3-2\epsilon}
  d{r},
\end{equation*}
and we have
\begin{align*}
\int_{0}^{\delta}\int_0^1
 \|\partial_{r} u^i \|_{L^2(U_{ij}^1)}^2 
 r^{2\alpha_i+1} t^{1-2\epsilon} dr dt
  \lesssim& \int_0^{1} 
  \|\partial_{r} f^i \|_{L^2(U_{ij}^1)}^2 
  r^{2\alpha_i + 3-2\epsilon}
  d{r}.
\end{align*}

We can similarly estimate
\begin{equation*}
 \int_{0}^{\delta}\int_0^1
 \|D_{\tau}^i u^i \|_{L^2(U_{ij}^1)}^2 
 r^{2\alpha_i+1}t^{1-2\epsilon} dr dt,
\end{equation*}
where $D_{\tau}^i$ is one of the
 last three vector fields in 
\eqref{vfEasy}.  For example,
with $D_{\tau}^i = r^{-\alpha_i} \partial_{x^i}$, we have
\begin{align*}
\int_{0}^{\delta}\int_0^1
\|D_{\tau}^i u^i \|_{L^2(U_{ij}^1)}^2 
r^{2\alpha_i+1}t^{1-2\epsilon} dr dt
 \simeq
\int_{0}^{\delta}\int_0^1
\|\partial_{x^i}u^i \|_{L^2(U_{ij}^1)}^2 
rt^{1-2\epsilon} dr dt
\end{align*}
and proceeding as above leads to
\begin{align*}
\int_{0}^{\delta}\int_0^1
\|\partial_{x^i}u^i \|_{L^2(U_{ij}^1)}^2 
rt^{1-2\epsilon} dr dt
\lesssim& \int_0^1
\|\partial_{x^i}f^i \|_{L^2(U_{ij}^1)}^2 
r r^{2-2\epsilon}  dr\\
\lesssim& 
 \int_0^1
\|r^{-\alpha_i} \partial_{x^i}f^i \|_{L^2(U_{ij}^1)}^2 
r^{2\alpha_i + 1} r^{2-2\epsilon}  dr.
\end{align*}

Thus the second integral on the right-hand side of
 \eqref{normalTangDecomp} 
  inserted into \eqref{t1-eNorm} 
   is bounded by
\begin{equation*}
\int_0^1
\|f^i\|_{W^1(U_{ij}^1)}^2 
r^{2\alpha_i+1} r^{2-2\epsilon} dr.
\end{equation*}

Similar estimates for the second integral on the 
 right hand side of \eqref{t1-eNorm} yield
\begin{equation*}
 \int_0^{\delta}
 t^{1-2\epsilon}
 \left\| u^i
 \right\|^2_{L^2(U_{ij})}  dt 
\lesssim 
 \|r^{1-\epsilon} f^i\|_{L^2(U_{ij})}^2 .
\end{equation*} 
 
Putting all this together in \eqref{t1-eNorm} yields
\begin{align*}
\int_0^{\infty}
 \big\| t^{1-\epsilon} u(t)
 \big\|^2_{W^1(U_{ij})} t^{-1}
 dt
 \lesssim&
 \int_0^1
\|\partial_r f^i \|_{L^2(U_{ij}^1)}^2 
r^{2\alpha_i+1}r^{2-2\epsilon} dr \\
&+ \int_0^1
\| f^i \|_{W^1(U_{ij}^1)}^2 
r^{2\alpha_i+1}r^{2-2\epsilon} dr\\
 \lesssim &
  \| r^{1-\epsilon} f \|_{W^1(U_{ij})}
   + \| r^{-\epsilon} f \|_{L^2(U_{ij})}. 
\end{align*}

In the case the metric is isometric
 to \eqref{NagaseMetric}, we proceed in a 
similar manner.   We have
\begin{equation*}
\sum_j \int_{{U}_{ij}} 
\bigg(
|L_k^i  u^i |^2 +|\Lb_k^i  u^i |^2
\bigg) dV^i 
\simeq 
 \sum_{k=1}^4 
\int_{{U}_{ij}}
 |D_k^i u^i |^2
dV^i ,
\end{equation*}
where $D_k^i$ are the vector fields in
 \eqref{vfNa}.
Let us handle the cases
$D_k^i = \frac{1}{r^{\alpha_i}}\partial_{s^i}$ and
 $D_k^i= \frac{1}{r^{\alpha_i}}
\frac{1}{r^{b_i}+s^i} \partial_{\varphi^i}$.  
 We first estimate
\begin{align*}
\int_0^{\delta} \int_0^1
\int
&
\left| \frac{1}{r^{\alpha_i}} \partial_{s^i} u^i\right|^2 
( r^{b_i}+s^i) r^{2\alpha_i + 1}  t^{1-2\epsilon} 
ds d\varphi^i d\theta^i
dr dt\\
  = &
\int_0^{\delta} \int_0^1
\int |\partial_{s^i} u^i|^2
 (r^{b_i}+s^i) r t^{1-2\epsilon} ds d\varphi^i
d\theta^i dr dt \\
=&
\int_0^{\delta} \int_t^1
\int |\partial_{s^i} u^i|^2
((r' - t) ^{b_i}+s^i) (r' - t) t^{1-2\epsilon}
 ds^i d\varphi^i d\theta^i dr'   dt.
\end{align*}
As above, we change the order of integration and 
 estimate
\begin{multline*}
\int
\int_0^{\delta} \int_0^{ r' }
|\partial_{s^i} u^i|^2
((r' - t) ^{b_i}+s^i) (r' - t)  t^{1-\epsilon} dt dr' 
ds^i d\varphi^i
 d\theta^i
 \\
 +
\int
\int_{\delta}^1 \int_0^{ \delta }
|\partial_{s^i} u^i|^2
((r' - t) ^{b_i}+s^i) (r' - t)  dt dr' 
ds^i d\varphi^i d\theta^i.
\end{multline*}
We use
\begin{multline*}
\int_0^{\delta} \int_0^{ r' }
|\partial_{s^i} u^i|^2
((r' - t) ^{b_i}+s^i)  (r' - t) 
 t^{1-\epsilon} dt dr' 
  \lesssim \\ 
\int_0^{\delta} 
|\partial_{s^i} f^i|^2
((r' ) ^{b_i}+s^i)  (r')^{3-2\epsilon}   dr' 
\end{multline*}
and
\begin{multline*}
\int_{\delta}^1 \int_0^{ \delta }
|\partial_{s^i} u^i|^2
((r' - t) ^{b_i}+s^i) (r' - t)  t^{1-\epsilon} dt dr' 
 \lesssim \\ 
\int_{\delta}^1 
|\partial_{s^i} f^i|^2
((r' ) ^{b_i}+s^i)  (r')^{3-2\epsilon}   dr' .
\end{multline*}

We have
\begin{align*}
\int_0^{\delta} \int_0^1
\int
\left| \frac{1}{r^{\alpha_i}} \partial_{s^i} u^i\right|^2 
( r^{b_i}+s^i) &r^{2\alpha_i + 1}  t^{1-2\epsilon} 
ds d\varphi^i d\theta^i
dr dt\\
\lesssim& 
\int \int_{0}^1 
 |\partial_{s^i} f^i|^2
 (r ^{b_i}+s^i)  r^{3-2\epsilon}   dr
 ds^i d\varphi^i
 d\theta^i \\
= &
\int \int_0^1
\left| \frac{1}{r^{\alpha_i}} \partial_{s^i} f^i\right|^2 
( r^{b_i}+s^i) r^{2\alpha_i + 3-2\epsilon} dr 
ds d\varphi^i d\theta^i\\
\simeq&
\int
\left| \frac{1}{r^{\alpha_i}} \partial_{s^i} f^i\right|^2 
r^{2-2\epsilon}  dV^i
 .
\end{align*}

Similarly, in the case
$D_k^i= \frac{1}{r^{\alpha_i}}
\frac{1}{r^{b_i}+s^i} \partial_{\varphi^i}$
we estimate
\begin{align*}
\int_0^{\delta} \int_0^1 \int
&
| \partial_{\varphi^i} u^i|^2 
( r^{b_i}+s^i)^{-1} r t^{1-2\epsilon} 
ds d\varphi^i d\theta^i
 dr  dt\\
=&
\int_0^{\delta} \int_t^1
\int 
|\partial_{\varphi^i} u^i|^2
\frac{ (r' - t)}{(r' - t) ^{b_i}+s^i}
t^{1-2\epsilon}
ds^i d\varphi^i
d\theta^i dr'   dt.
\end{align*}
As above, we change the order of integration and 
estimate
\begin{multline*}
\int \int_0^{\delta} \int_0^{ r' }
|\partial_{\varphi^i} u^i|^2
\frac{ (r' - t)}{(r' - t) ^{b_i}+s^i}
t^{1-2\epsilon}
dt dr' 
ds^i d\varphi^i
d\theta^i
\\
+
\int \int_{\delta}^1 \int_0^{ \delta }
|\partial_{\varphi^i} u^i|^2
\frac{ (r' - t)}{(r' - t) ^{b_i}+s^i}
t^{1-2\epsilon}
 dt dr' 
ds^i d\varphi^i d\theta^i.
\end{multline*}
We use
\begin{equation*}
\frac{ (r' - t)}{(r' - t) ^{b_i}+s^i} 
 \lesssim 
\frac{ r' }{(r') ^{b_i}+s^i}
\end{equation*}
(for $r'>t$ and $0<b_i<1$)
in both integrals to estimate
\begin{align*}
\int_0^{\delta} \int_0^1 \int
| \partial_{\varphi^i} u^i|^2 
( r^{b_i}+s^i)^{-1} &r t^{1-2\epsilon} 
ds d\varphi^i d\theta^i
dr  dt
 \\
\lesssim  & 
\int \int_{0}^1 
|\partial_{\varphi^i} f^i|^2
\frac{ r}{r^{b_i}+s^i}
r^{2-2\epsilon}
dr
ds^i d\varphi^i d\theta^i\\
  \lesssim&
\int
\left| \frac{1}{r^{\alpha_i}} \frac{1}{(r^{b_i} +s^i)}  \partial_{\varphi^i} f^i \right|^2
r^{2-2\epsilon}
dV^i.
\end{align*}

The other vector fields in \eqref{vfNa} are handled 
 similarly and we obtain in the case the metric over
$U_i$ is quasi-isometric to \eqref{NagaseMetric}
\begin{equation*}
\int_0^{\infty}
\big\| t^{1-\epsilon} u(t)
\big\|^2_{W^1(U_{ij})} t^{-1}
dt
\lesssim 
\| r^{1-\epsilon} f \|_{W^1(U_{ij})}
+ \| r^{-\epsilon} f \|_{L^2(U_{ij})}
\end{equation*}
as above.

For the second integral in \eqref{defnIu}
we use
\begin{align*}
\partial_t u^i(t)=& (\partial_t f^i_t )\eta(t) + f^i_t \eta'(t)\\
=& (\partial_{r} f^i_t ) \eta(t) 
+ f^i_t \eta'(t).
\end{align*}
Therefore,
\begin{multline}
\label{secondL2Int}
\int_0^{\infty} \left\| t^{1-\epsilon} \partial_t u^i(t)
\right\|^2_{L^2(X)} t^{-1}
dt
\lesssim 
\sum_{ij}  \int_0^{\delta}
 \int_{{U}_{ij}} 
 | \partial_r f^i_t |^2 t^{1-2\epsilon} dV^i dt\\
 +\sum_{ij} \int_0^{\delta}
\int_{U_{ij}} |  f^i_t |^2
t^{1-2\epsilon} dV^i dt.
\end{multline}
Using a change of coordinates
 $r=r' +t$ as above
 in the first integral above yields
\begin{align*}
\int_0^{\delta}
  \int_{U_{ij}} 
| \partial_r f^i_t|^2 t^{1-2\epsilon} dV^i dt
  \lesssim& \int 
 |\partial_{r} f^i |^2 
 r^{2-2\epsilon}
 dV^i\\
 \lesssim&  \| r^{1-\epsilon} f^i \|_{W^1(U_{ij})}
 + \| r^{-\epsilon} f^i \|_{L^2(U_{ij})}. 
\end{align*}
Summing over $U_{ij}$ shows the first integral on the
 right-hand side of \eqref{secondL2Int} is bounded by
\begin{equation*}
 \| r^{1-\epsilon} f\|_{W^1(X)} +
 \| r^{-\epsilon} f\|_{L^2(X)}.
\end{equation*}

 Estimates for the second term on the right of
  \eqref{secondL2Int} follow as 
those above, and
in terms of integrals over $X$, we have
\begin{equation*}
\int_0^{\delta}
\left\| t^{1-\epsilon} \partial_t u^i(t)
\right\|^2_{L^2(X)} t^{-1}
dt  
\lesssim
\| r^{1-\epsilon} f\|_{W^1(X)} +
\| r^{-\epsilon} f\|_{L^2(X)}.
\end{equation*}

\end{proof}

\section{Approximation by smooth forms}

Theorem \ref{ThrmSubellIneq}
was proved under the condition of smoothness 
of the function to be
 estimated.  We use an approximation argument in this
section to broaden the class of functions (or forms)
 to which the theorem applies. 

Let us define the weighted $L^{2}$ norms
\begin{equation*}
 L^{2,-C}(X):= \{f\in L^2(X) : r^{-C} f \in L^2(X) \}
\end{equation*}
for $C>1$.  The motivation for the
 weights, $r^{-C}$, for $C>1$ comes from the 
$L^2$ norms on the right-hand side of 
 \eqref{basicEstLog}.
Recall that weighted norms for forms are defined by
 the norm for functions applied component-wise.
Thus, for instance, $L^{2,-C}_{(p,q)}(X)$ is defined 
 to consist of those $(p,q)$-forms such that
each component is in $L^{2,-C}(X)$.
 
We use the cutoffs $\mu_k$ from Section \ref{sectionSobolev1}.
Recall
$\mu_k = \mu_k(|z|)$ is a smooth function
such that $\mu_k = 0$ for 
$|z| <   e^{-e^{k+1}}$ and 
$\mu_k = 1$ for $|z| > e^{-e^{k}}$,
 and with the property
\begin{equation}
\label{dmu}
| d \mu_k(z) | \lesssim \frac{ \chi_k (|z|^2)}
{|z| \log |z| },
\end{equation}
where $\chi_k$ is the characteristic function of
$[e^{-e^{k+1}},e^{-e^{k}}]$.  With
$f_k : = \mu_k f$ we have
\begin{lemma}
	\label{gdf}
	Let $C>1$ and $f$ with support near the origin
	be such that
	$ f \in L^{2,-C}_{(p,q)}(X) $,
	$\mdbar f\in L^2_{(p,q+1)}(X)$, and $\mdbar^{\ast} f\in L^2_{(p,q-1)}(X)$. Then
	\begin{equation}
	\label{rclim}
	r^{-C}f_{k}
	\rightarrow r^{-C} f \ \ \ \mbox{ in } \ \ L^2_{(p,q)}(X),
	\end{equation}
	and
	\begin{equation}
		\label{difflim}
	\begin{aligned}
	\mdbar f_k \rightarrow  \mdbar f & \mbox{ in } L^2_{(p,q+1)}(X)\\
	\mdbar^{\ast} f_k \rightarrow  \mdbar^{\ast} f &\mbox{ in }  L^2_{(p,q-1)} (X)
	.
	\end{aligned}
	\end{equation}
	Furthermore,
	\begin{equation}
	\label{w1lim}
	f_k \overset{W^1}{\rightarrow} f.
	\end{equation}
\end{lemma}
\begin{proof}
	The limit \eqref{rclim} is clear.  To show \eqref{difflim} we
	estimate
	\begin{align*}
	\| \mdbar \mu_k f \|_{L^2(X)} \le &  
	\| \mu_k \mdbar f\|_{L^2(X)} + \|\mdbar \mu_k \wedge f\|_{L^2(X)} \\
	\lesssim & \| \mu_k \mdbar f\|_{L^2(X)}
	  + \left\|\frac{\chi_k(r^2)}{r\log(r)}  
	   f_k\right\|_{L^2(X )},
	\\
	\lesssim & \| \mu_k \mdbar f\|_{L^2(X)}
	+ \left\|r^{-C}
	f_k\right\|_{L^2(X )}
	\end{align*}
	using \eqref{dmu}.
	
Now letting $k\rightarrow \infty$ we have
	$\| \mu_k \mdbar f\|_{L^2} \rightarrow  \| \mdbar f\|_{L^2} $ 
from the dominated convergence theorem. 
	Thus $\|  \mdbar f_k\|_{L^2} \rightarrow  \| \mdbar f\|_{L^2}$
	 and a similar
	argument shows $\mdbar^{\ast} f_k \overset{L^2}{\rightarrow}
	\mdbar^{\ast} f$.
When combined with 
 \eqref{basicEstLog}, \eqref{rclim} and \eqref{difflim}  yield \eqref{w1lim}.
\end{proof}

In Section \ref{sectionSobolev1}, we showed
\begin{equation*}
\| \varphi \|_{W_1(X)}^2 \lesssim \|\mdbar \varphi \|^2_{L^2(X)} +
\| \mdbar^{\ast}\varphi \|^2_{L^2(X)} + 
 \left\|  \frac{1}{r\log(r)}\varphi \right\|_{L^2(X)}
\end{equation*}
for smooth forms, $\varphi$.  
Now with $f \in L^{2,-C}_{(p,q)}(X)$ for $C>1$, 
$\mdbar f\in L^2_{(p,q+1)}(X)$, and 
$\mdbar^{\ast} f\in L^2_{(p,q-1)}(X)$,
and with
 $f_k = \mu_k f$ as above, we define smooth
$f_k^{\delta}$ with the use of mollifiers
 so that 
\begin{equation*}
 \begin{aligned}
	  r^{-C}f_k^{\delta}
	 \overset{L^2}{\rightarrow} & r^{-C} f_k\\
  	\mdbar f_k^{\delta}
  	\overset{L^2}{\rightarrow} & \mdbar f_k\\
  	\mdbar^{\ast} f_k^{\delta}
  	\overset{L^2}{\rightarrow} & \mdbar^{\ast} f_k
  \end{aligned}
\end{equation*}
(as $\delta\rightarrow 0$).

 Letting $\varphi = f_k^{\delta}$ in the above estimates
 and letting $\delta\rightarrow 0$,
gives
\begin{equation*}
\| f_k \|_{W^1(X)}^2 \lesssim \|\mdbar f_k\|^2_{L^2(X)} +
\| \mdbar^{\ast}f_k\|^2_{L^2(X)} + 
 \left\|  \frac{1}{r\log(r)} f_k  \right\|_{L^2(X)}.
\end{equation*}


We can now apply Lemma
\ref{gdf} and let $k\rightarrow \infty$.  We obtain
for $f\in L^{2,-C}_{(p,q)}(X)$ ($0\le p, q\le 2$), with
$\mdbar f\in L^2_{(p,q+1)}(X)$, and 
$\mdbar^{\ast} f\in L^2_{(p,q-1)}(X)$,
that $f\in W^1_{(p,q)}(X)$ with estimates
	\begin{equation}
		\label{w1est}
	\|f\|_{W^1(X)}
	\lesssim \| \mdbar f\|^2_{L^2(X)} + \| \mdbar^{\ast} f\|^2_{L^2(X)}
	+ \left\|  \frac{1}{r\log(r)} f \right\|_{L^2(X)}.
	\end{equation}

Combining \eqref{w1est} with
Theorem \ref{ThrmSubellIneq}, we have
\begin{thrm}
	\label{ThrmSobEst}
	 Let $C>1$.
For
$f$ supported near the origin such that
$ f \in L^{2,-C}_{(p,q)}(X) $
and 
 $ f\in L^2_{(p,q)}(X)
 \cap \mbox{dom}(\mdbar)
 \cap \mbox{dom}(\mdbar^{\ast})$
 we have
\begin{equation}
\nonumber
\|f\|_{W^{\epsilon}(X)}
\label{r1repeps}
\lesssim
\left\|  r^{1-\epsilon} \mdbar f\right\|_{L^2(X)} +
\left\| r^{1-\epsilon}\mdbar^{\ast} f\right\|_{L^2(X)}
+\left\|  \frac{1}{r^{\epsilon} \log(r)} f \right\|_{L^2(X)}
\end{equation}
 for 
$0<\epsilon \le 1$. 
\end{thrm}
If we fix $\epsilon$, we can replace the 
hypothesis that
 $f\in L^{2,-C}_{(p,q)}(X)$ with $\frac{1}{r^{\epsilon} \log(r)} f  \in L^{2}_{(p,q)}(X) $
  as in the Main Theorem.  The theorem can easily be extended to 
the case of forms whose support contains 
 multiple singularities of $X$.

As a final corollary we relate 
 $\left\|  \frac{1}{r^{\epsilon} \log(r)} f \right\|_{L^2(X)}$
  estimates to $L^p$-estimates.
With 
\begin{equation*}
\frac{1}{p}+\frac{1}{p'}=1
\end{equation*}
we have for a function, $f\in L^{2p}(X)$ with support near 0,
\begin{align*}
\left\|  \frac{1}{r^{\epsilon} \log(r)} f \right\|^2_{L^2(X)}
  \lesssim& \|  r^{-\delta} f \|^2_{L^2(X)} \\
  \lesssim& \sum_i
    \left(\int_{U_i} |r^{-2p'\delta}| dV^i
    \right)^{1/p'} \left(\int_{U_i} | f^i |^{2p} dV^i
    \right)^{1/p}
\end{align*}
for $\delta > \epsilon$.
Using
 \eqref{volumeHP} for each
$U_i$ the integration of 
 $r^{-2p'\delta}$ converges for
$p'<2/\delta$, which corresponds to
 $p> 2/(2-\delta)$.
 
We apply the above with a $\delta>1$.
We conclude that with $f\in L^{2p}(X)  $ for
 $p> 2$, we have that $f$ is also in $L^{2,-C}(X)$
for some $C>1$.  If in addition,
 $\mdbar f\in L^2(X)$, and $\mdbar^{\ast} f\in L^2(X)$,
then Theorem \ref{ThrmSobEst} applies, and,
 since $\|r^{-\epsilon}f\| \lesssim \|r^{-\delta}f\|$, we have the
following
\begin{cor}
For a form,
$f$, supported near the origin such that
$ f \in L^{p}(X) \cap
\mbox{dom}(\mdbar)
\cap \mbox{dom}(\mdbar^{\ast}) $ with $p> 4$,
\begin{equation*}
\nonumber
\|f\|_{W^{\epsilon}(X)}
\lesssim
\left\|  r^{1-\epsilon} \mdbar f\right\|_{L^2(X)} +
\left\| r^{1-\epsilon}\mdbar^{\ast} f\right\|_{L^2(X)}
+\|f\|_{L^p(X)}
\end{equation*}
 for 
$0<\epsilon \le 1$. 
\end{cor}


\begin{thebibliography}{10}
	
	\bibitem{BL}
	J.~Bergh and J.~L{\"o}fstr{\"o}m.
	\newblock {\em Interpolation Spaces, an Introduction}.
	\newblock Springer-Verlag, 1976.
	
	\bibitem{CS}
	S.~Chen and M.~Shaw.
	\newblock {\em Partial Differential Equations in Several Complex Variables}.
	\newblock AMS/IP Studies in Advanced Mathematics. American Mathematical Society
	and International Press, 2001.
	
	\bibitem{ER}
	D.~Ehsani and J.~Ruppenthal.
	\newblock Subelliptic estimates for the $\bar{\partial}$-problem on a singular
	complex space.
	\newblock {\em J. Geom. Anal.}, 24(4):1844--1859, 2014.
	
	\bibitem{HsPa}
	W.-C. Hsiang and V.~Pati.
	\newblock ${L}^2$-cohomology of normal algebraic surfaces. {I}.
	\newblock {\em Invent. Math.}, 81(3):395--412, 1985.
	
	\bibitem{Lau}
	H.~Laufer.
	\newblock {\em Normal two dimensional singularities}, volume~71 of {\em Annals
		of Mathematics Studies}.
	\newblock Princeton University Press, Princeton, 1971.
	
	\bibitem{LiMa}
	J.-L. Lions and E.~Magenes.
	\newblock {\em Non-homogeneous boundary value problems and applications},
	volume~I.
	\newblock Springer-Verlag, New York, 1972.
	
	\bibitem{Na88}
	M.~Nagase.
	\newblock On the heat operators of normal singular algebraic surfaces.
	\newblock {\em J. Differ. Geom.}, 28(1):37--57, 1989.
	
	\bibitem{Na89}
	M.~Nagase.
	\newblock Remarks on the {$L^2$}-cohomology of singular algebraic surfaces.
	\newblock {\em J. Math. Soc. Japan}, 41(1):97--116, 1989.
	
	\bibitem{OV13}
	N.~{\O}vrelid and S.~Vassiliadou.
	\newblock {$L^2$}-$\bar{\partial}$-cohomology groups of some singular complex
	spaces.
	\newblock {\em Invent. Math.}, 192(2):413--458, 2013.
	
	\bibitem{PS91}
	W.~Pardon and M.~Stern.
	\newblock {$L^2$}-$\bar{\partial}$-cohomology of complex projective varieties.
	\newblock {\em J. Amer. Math. Soc.}, 4(3):603--621, 1991.
	
	\bibitem{Rup19}
	J.~Ruppenthal.
	\newblock {$L^2$}-theory for the $\bar{\partial}$--operator on complex spaces
	with isolated singularities.
	\newblock {\em Ann. Fac. Sci. Toulouse, Math. (6)}, 28(2):225--258, 2019.
	
\end{thebibliography}
\end{document}